\documentclass{amsart}
\usepackage[latin1]{inputenc}

\usepackage{amsfonts}
\usepackage{amssymb,amsmath}
\usepackage{amsthm}
\usepackage{paralist}
\usepackage{url}

\usepackage[normalem]{ulem}

\theoremstyle{plain}
\newtheorem{theorem}{Theorem}[section]
\newtheorem{proposition}[theorem]{Proposition}
\newtheorem{lemma}[theorem]{Lemma}

\theoremstyle{definition}

\theoremstyle{remark}
\newtheorem{remark}[theorem]{Remark}

\newcommand{\vect}[1]{\ensuremath{\mathbf{#1}}} 
\newcommand{\card}[1]{\ensuremath{\lvert{#1}\rvert}} 
\DeclareMathOperator{\range}{Im} 
\DeclareMathOperator{\ess}{ess} 
\DeclareMathOperator{\essl}{ess^{<}} 
\DeclareMathOperator{\gap}{gap} 
\DeclareMathOperator{\qa}{qa} 

\newcommand{\oddsupp}{\ensuremath{\mathrm{oddsupp}}} 

\newcommand{\Aneq}[1][n]{\ensuremath{A^{#1}_{=}}}

\begin{document}
\title{Decompositions of functions based on arity gap}

\author{Miguel Couceiro}
\address[M. Couceiro]{Mathematics Research Unit \\
University of Luxembourg \\
6, rue Richard Coudenhove-Kalergi \\
L-1359 Luxembourg \\
Luxembourg}
\email{miguel.couceiro@uni.lu}

\author{Erkko Lehtonen}
\address[E. Lehtonen]{Computer Science Research Unit \\
University of Luxembourg \\
6, rue Richard Coudenhove-Kalergi \\
L-1359 Luxembourg \\
Luxembourg}
\email{erkko.lehtonen@uni.lu}

\author[Tamás Waldhauser]{Tamás Waldhauser$^\text{1}$}
\address[T. Waldhauser]{Mathematics Research Unit \\
University of Luxembourg \\
6, rue Richard Coudenhove-Kalergi \\
L-1359 Luxembourg \\
Luxembourg
\and
Bolyai Institute \\
University of Szeged \\
Aradi v\'{e}rtan\'{u}k tere 1 \\
H-6720 Szeged \\
Hungary}
\email{twaldha@math.u-szeged.hu}
\thanks{$^\text{1}$The present project is supported by the National Research Fund, Luxembourg, and cofunded under the Marie Curie Actions of the European Commission (FP7-COFUND), and supported by the Hungarian National Foundation for Scientific Research under grant no.\ K77409.}

\date{\today}

\keywords{Arity gap, variable identification minors, Boolean groups}

\subjclass[2000]{Primary: 08A40}

\begin{abstract}
We study the arity gap of functions of several variables defined on an arbitrary set $A$ and valued in another set $B$. The arity gap of such a function is the minimum decrease in the number of essential variables when variables are identified. We establish a complete classification of functions according to their arity gap, extending existing results for finite functions. This classification is refined when the codomain $B$ has a group structure, by providing unique decompositions into sums of functions of a prescribed form. As an application of the unique decompositions, in the case of finite sets we count, for each $n$ and $p$, the number of $n$-ary functions that depend on all of their variables and have arity gap $p$.
\end{abstract}

\maketitle


\section{Introduction}

Essential variables of functions have been investigated in multiple-valued logic and computer science, especially, concerning the distribution of values of functions whose variables are all essential (see, e.g., \cite{Davies,Salomaa,Yablonski}), the process of substituting constants for variables (see, e.g., \cite{Cimev1981,Cimev1986,Lupanov,Salomaa,Solovjev}), and the process of substituting variables for variables (see, e.g., \cite{CL2007,DK,Salomaa,Willard}).

The latter line of study goes back to the 1963 paper by Salomaa~\cite{Salomaa} who considered the following problem: How does identification of variables affect the number of essential variables of a given function? The minimum decrease in the number of essential variables of a function $f \colon A^n \to B$ ($n \geq 2$) which depends on all of its variables is called the arity gap of $f$. Salomaa \cite{Salomaa} showed that the arity gap of any Boolean function is at most $2$. This result was extended to functions defined on arbitrary finite domains by Willard \cite{Willard}, who showed that the same upper bound holds for the arity gap of any function $f \colon A^n \to B$, provided that $n > \card{A}$. In fact, he showed that if the arity gap of such a function $f$ is $2$, then $f$ is totally symmetric. 
Salomaa's \cite{Salomaa} result on the upper bound for the arity gap of Boolean functions was strengthened in~\cite{CL2007}, where Boolean functions were completely classified according to their arity gap. In~\cite{CL2009}, by making use of tools provided by Berman and Kisielewicz \cite{BK} and Willard \cite{Willard}, a similar explicit classification was obtained for all pseudo-Boolean functions, i.e., functions $f \colon \{0,1\}^n \to B$, where $B$ is an arbitrary set. This line of study culminated in a complete classification of functions $f \colon A^n \to B$ with finite domains according to their arity gap in terms of so-called quasi-arity; see Theorem~\ref{thm:gap}, first presented in~\cite{CL2009}.

Although Theorem~\ref{thm:gap} was originally stated in the setting of functions $f \colon A^n \to B$ with finite domains, it actually holds for functions with arbitrary, possibly infinite domains (see Remark~\ref{rem:generalization} in Section~\ref{sec:gap}). Alas, this classification is not quite explicit. However, as we will see in Section~\ref{sec:decomposition}, provided that the codomain $B$ has a group structure, this classification can be refined to a unique decomposition of functions as a sum of functions of a prescribed type (see Theorem~\ref{thm:sum}). This result can be further strengthened by assuming that $B$ is a Boolean group (see Section~\ref{sec:gap2}). As an application of the unique decomposition theorem, in Section~\ref{sec:enumerate}, assuming that sets $A$ and $B$ are finite, we will count for each $n$ and $p$ the number of functions $f \colon A^n \to B$ that depend on all of their variables and have arity gap $p$.


\section{Essential arity and quasi-arity}
\label{sec:minors}

Throughout this paper, let $A$ and $B$ be arbitrary sets with at least two elements. A \emph{$B$-valued function} (\emph{of several variables}) \emph{on $A$} is a mapping $f \colon A^n \to B$ for some positive integer $n$, called the \emph{arity} of $f$. $A$-valued functions on $A$ are called \emph{operations on $A$.} Operations on $\{0,1\}$ are called \emph{Boolean functions.} For an arbitrary $B$, we refer to $B$-valued functions on $\{0, 1\}$ as \emph{pseudo-Boolean functions.}

A \emph{partial function} from $X$ to $Y$ is a map $f \colon S \to Y$ for some $S \subseteq X$. In the case that $S = X$, we speak of \emph{total functions.} Thus, an $n$-ary \emph{partial function} from $A$ to $B$ is a map $f \colon S \to B$ for some $S \subseteq A^n$.

Let $f \colon S \to B$ be a partial function with $S \subseteq A^n$. We say that the $i$-th variable is \emph{essential} in $f$, or $f$ \emph{depends} on $x_i$, if there is a pair
\[
((a_1, \ldots, a_{i-1}, a_i, a_{i+1}, \ldots, a_n), (a_1, \ldots, a_{i-1}, b, a_{i+1}, \ldots, a_n)) \in S^2,
\]
called a \emph{witness of essentiality} of $x_i$ in $f$, such that
\[
f(a_1, \ldots, a_{i-1}, a_i, a_{i+1}, \ldots, a_n) \neq f(a_1, \ldots, a_{i-1}, b, a_{i+1}, \ldots, a_n).
\]
The number of essential variables in $f$ is called the \emph{essential arity} of $f$, and it is denoted by $\ess f$. If $\ess f = m$, we say that $f$ is \emph{essentially $m$-ary.} Note that the only essentially nullary total functions are the constant functions, but this does not hold in general for partial functions.

For $n \geq 2$, define
\[
\Aneq := \{(a_1, \dots, a_n) \in A^n : \text{$a_i = a_j$ for some $i \neq j$}\}.
\]
We also define $\Aneq[1] := A$. Note that if $A$ has less than $n$ elements, then $\Aneq = A^n$.

\begin{lemma}
\label{lem:diagwitness}
Let $f \colon A^n \to B$, $n \geq 3$, $\ess f < n$. Then for each essential variable $x_i$, there exists a pair of points $(\vect{a}, \vect{b}) \in (\Aneq)^2$ that is a witness of essentiality of $x_i$ in $f$.
\end{lemma}

\begin{proof}
Since $\ess f < n$, $f$ has an inessential variable. Assume, without loss of generality, that $x_n$ is inessential in $f$. Assume that $x_i$ is an essential variable in $f$, and let
\[
((a_1, \ldots, a_{i-1}, a_i, a_{i+1}, \ldots, a_n), (a_1, \ldots, a_{i-1}, b, a_{i+1}, \ldots, a_n)) \in (A^n)^2
\]
be a witness of essentiality of $x_i$ in $f$. Let $j \in \{1, \dots, n-1\} \setminus \{i\}$. We have that
\begin{multline*}
f(a_1, \ldots, a_{i-1}, a_i, a_{i+1}, \ldots, a_{n-1}, a_j)
= f(a_1, \ldots, a_{i-1}, a_i, a_{i+1}, \ldots, a_{n-1}, a_n) \\
\neq f(a_1, \ldots, a_{i-1}, b, a_{i+1}, \ldots, a_{n-1}, a_n)
= f(a_1, \ldots, a_{i-1}, b, a_{i+1}, \ldots, a_{n-1}, a_j),
\end{multline*}
where the two equalities hold by the assumption that $x_j$ is inessential in $f$, and the inequality holds by our choice of a witness of essentiality of $x_i$ in $f$.
Thus,
\[
((a_1, \ldots, a_{i-1}, a_i, a_{i+1}, \ldots, a_{n-1}, a_j), (a_1, \ldots, a_{i-1}, b, a_{i+1}, \ldots, a_{n-1}, a_j)) \in (\Aneq)^2
\]
is a witness of essentiality of $x_i$ in $f$.
\end{proof}

We say that a function $f \colon A^n \to B$ is obtained from $g \colon A^m \to B$ by \emph{simple variable substitution,} or $f$ is a \emph{simple minor} of $g$, if there is a mapping $\sigma \colon \{1, \ldots, m\} \to \{1, \ldots, n\}$ such that
\[
f(x_1, \dots, x_n) = g(x_{\sigma(1)}, \ldots, x_{\sigma(m)}).
\]
If $\sigma$ is not injective, then we speak of \emph{identification of variables.} If $\sigma$ is not surjective, then we speak of \emph{addition of inessential variables.} If $\sigma$ is a bijection, then we speak of \emph{permutation of variables.} 

The simple minor relation constitutes a quasi-order $\leq$ on the set of all $B$-valued functions of several variables on $A$ which is given by the following rule: $f \leq g$ if and only if $f$ is obtained from $g$ by simple variable substitution. If $f \leq g$ and $g \leq f$, we say that $f$ and $g$ are \emph{equivalent,} denoted $f \equiv g$. If $f \leq g$ but $g \not\leq f$, we denote $f < g$. It can be easily observed that if $f \leq g$ then $\ess f \leq \ess g$, with equality if and only if $f \equiv g$. For background, extensions and variants of the simple minor relation, see, e.g., \cite{Couceiro,CP,FH,Lehtonen,LS,Pippenger,Wang,Zverovich}.

Consider $f \colon A^n \to B$. Any function $g \colon A^n \to B$ satisfying $f|_{\Aneq} = g|_{\Aneq}$ is called a \emph{support} of $f$. The \emph{quasi-arity} of $f$, denoted $\qa f$, is defined as the minimum of the essential arities of the supports of $f$, i.e., $\qa f = \min_g \ess g$, where $g$ ranges over the set of all supports of $f$. If $\qa f = m$, we say that $f$ is \emph{quasi-$m$-ary.} 

The following two lemmas were proved in~\cite{CL2009}.

\begin{lemma}
\label{lem:qaess}
For every function $f \colon A^n \to B$, $n \neq 2$, we have $\qa f = \ess f|_{\Aneq}$.
\end{lemma}

\begin{lemma}\label{lem:m-ary}
If any variable of a quasi-$m$-ary function $f \colon A^n \to B$ is inessential, then $f$ is essentially $m$-ary.
\end{lemma}

\begin{remark}
\label{rem:an}
If $A$ is a finite set and $n > \card{A}$, then $\Aneq = A^n$, and hence for every $f \colon A^n \to B$ we have $\qa f = \ess f$.
\end{remark}

In the sequel, we will make use of the following result.

\begin{proposition}
\label{prop:uniqueSupport}
Let $f \colon A^n \to B$, $n \geq 3$. If $\ess f = n > m = \qa f$, then $f$ has a unique essentially $m$-ary support.
\end{proposition}

\begin{proof}
Let $g \colon A^n \to B$ be an essentially $m$-ary support of $f$, say, with $x_1, \ldots, x_m$ essential. By Lemma~\ref{lem:diagwitness}, $g$ and $f|_{\Aneq}$ have the same essential variables. Now if $h \colon A^n \to B$ is an essentially $m$-ary support of $f$, then $x_1, \ldots, x_m$ are exactly the essential variables of $h$, and
\begin{align*}
h(x_1, \ldots, x_n)
& = h(x_1, \ldots, x_m, x_m, \ldots, x_m)
  = f(x_1, \ldots, x_m, x_m, \ldots, x_m) \\
& = g(x_1, \ldots, x_m, x_m, \ldots, x_m)
  = g(x_1, \ldots, x_n).
\end{align*}
Thus $h$ and $g$ coincide.
\end{proof}


\section{Arity gap}
\label{sec:gap}

Recall that simple variable substitution induces a quasi-order on the set of $B$-valued functions on $A$, as described in Section \ref{sec:minors}. For a function $f \colon A^n \to B$ with at least two essential variables, we denote
\[
\essl f = \max_{g < f} \ess g,
\]
and we define the \emph{arity gap} of $f$ by $\gap f = \ess f - \essl f$.

In the sequel, whenever we consider the arity gap of some function $f$, we will assume that all variables of $f$ are essential. This is not a significant restriction, because every non-constant function is equivalent to a function with no inessential variables and equivalent functions have the same arity gap.

Salomaa~\cite{Salomaa} proved that the arity gap of every Boolean function with at least two essential variables is at most $2$. This result was generalized by Willard \cite[Lemma 1.2]{Willard} in the following theorem.

\begin{theorem}
\label{Willard1.2}
Let $A$ be a finite set. Suppose $f \colon A^n \to B$ depends on all of its variables. If $n > \card{A}$, then $\gap f \leq 2$.
\end{theorem}

In \cite{CL2007}, Salomaa's result was strengthened by completely classifying all Boolean functions in terms of arity gap: for $f \colon \{0, 1\}^n \to \{0, 1\}$, $\gap f = 2$ if and only if $f$ is equivalent to one of the following Boolean functions:
\begin{compactitem}
\item $x_1 + x_2 + \dots + x_m + c$,
\item $x_1 x_2 + x_1 + c$,
\item $x_1 x_2 + x_1 x_3 + x_2 x_3 + c$,
\item $x_1 x_2 + x_1 x_3 + x_2 x_3 + x_1 + x_2 + c$,
\end{compactitem}
where addition and multiplication are done modulo 2 and $c \in \{0, 1\}$. Otherwise $\gap f = 1$.

Based on this, a complete classification of pseudo-Boolean functions according to their arity gap was presented in~\cite{CL2009}.

\begin{theorem}
\label{gappseudoB}
For a pseudo-Boolean function $f \colon \{0,1\}^n \to B$ which depends on all of its variables, $\gap f = 2$ if and only if $f$ satisfies one of the following conditions:
\begin{compactitem}
\item $n = 2$ and $f$ is a nonconstant function satisfying $f(0,0) = f(1,1)$,
\item $f = g \circ h$, where $g \colon \{0, 1\} \to B$ is injective and $h \colon \{0, 1\}^n \to \{0, 1\}$ is a Boolean function with $\gap h = 2$, as listed above.
\end{compactitem}
Otherwise $\gap f = 1$.
\end{theorem}

The study of the arity gap of functions $A^n \to B$ culminated into the characterization presented in Theorem~\ref{thm:gap}, originally proved in~\cite{CL2009}. We need to introduce some terminology to state the result. Denote by $\mathcal{P}(A)$ the power set of $A$, and define the function $\oddsupp \colon \bigcup_{n \geq 1} A^n \to \mathcal{P}(A)$ by
\[
\oddsupp(a_1, \dots, a_n) =
\{a_i : \text{$\card{\{j \in \{1, \dots, n\} : a_j = a_i\}}$ is odd}\}.
\]
We say that a partial function $f \colon S \to B$, $S \subseteq A^n$, is \emph{determined by $\oddsupp$} if $f = f^* \circ \, \oddsupp|_S$ for some function $f^* \colon \mathcal{P}(A) \to B$. In order to avoid cumbersome notation, if $f \colon S \to B$, $S \subseteq A^n$, is determined by $\oddsupp$, then whenever we refer to the decomposition $f = f^* \circ \oddsupp|_S$, we may write simply ``$\oddsupp$'' in place of ``$\oddsupp|_S$'', omitting the subscript indicating the domain restriction as it will be obvious from the context.

\begin{remark}
The notion of a function being determined by $\oddsupp$ is due to Berman and Kisielewicz~\cite{BK}. Willard~\cite{Willard} showed that if $f \colon A^n \to B$ where $A$ is finite, $n > \max(\lvert A \rvert, 3)$ and $\gap f = 2$, then $f$ is determined by $\oddsupp$.
\end{remark}

\begin{remark}
\label{rem:oddsupp}
It is easy to verify that for $n \geq 2$,
\[
\range \oddsupp|_{\Aneq} = \{S \subseteq A : \card{S} \equiv n \!\!\!\!\pmod{2},\, \card{S} \leq n - 2\}.
\]
Thus, if $f \colon \Aneq \to B$ is determined by oddsupp, i.e., $f = f^* \circ \oddsupp|_{\Aneq}$, then within the domain $\mathcal{P}(A)$ of $f^*$, only the subsets of $A$ of order at most $n - 2$ with the same parity as $n$ (odd or even) are relevant.
\end{remark}

\begin{remark}
\label{rem:oddsupp2}
A function $f \colon A^n \to A$ is determined by $\oddsupp$ if and only if $f|_{\Aneq}$ is determined by $\oddsupp$ and $f$ is totally symmetric.
\end{remark}

\begin{theorem}
\label{thm:gap}
Let $A$ and $B$ be arbirary sets with at least two elements.
Suppose that $f \colon A^n \to B$, $n \geq 2$, depends on all of its variables.
\begin{compactenum}[\indent\rm (i)]
\item \label{thmitem1} For $3 \leq p \leq n$, $\gap f = p$ if and only if $\qa f = n - p$.
\item \label{thmitem2} For $n \neq 3$, $\gap f = 2$ if and only if $\qa f = n - 2$ or $\qa f = n$ and $f|_{\Aneq}$ is determined by $\oddsupp$.
\item \label{thmitem3} For $n = 3$, $\gap f = 2$ if and only if there is a nonconstant unary function $h \colon A \to B$ and $i_1, i_2, i_3 \in \{0,1\}$ such that
\begin{align*}
f(x_1, x_0, x_0) &= h(x_{i_1}), \\
f(x_0, x_1, x_0) &= h(x_{i_2}), \\
f(x_0, x_0, x_1) &= h(x_{i_3}).
\end{align*}
\item Otherwise $\gap f = 1$.
\end{compactenum}
\end{theorem}

\begin{remark}
\label{rem:generalization}
While Theorem~\ref{thm:gap} was originally presented in the setting of functions $f \colon A^n \to B$ where $A$ is a finite set, its proof does not make use of any assumption on the cardinality of $A$ -- except for $A$ having at least two elements -- so it immediately generalizes to functions with arbitrary domains. 
\end{remark}

\section{A decomposition theorem for functions}
\label{sec:decomposition}

In this section, we will establish the following classification of functions $f \colon A^n \to B$ ($n \geq 3$) with arity gap $p \geq 3$, which also provides a decomposition of such functions into a sum of a quasi-nullary function and an essentially $(n - p)$-ary function.

\begin{theorem}
\label{thm:sum}
Assume that $(B; +)$ is a group with neutral element $0$. Let $f \colon A^n \to B$, $n \geq 3$, and $3 \leq p \leq n$. Then the following two conditions are equivalent:
\begin{enumerate}[\indent\rm (1)]
\item \label{itemthmsum1}
$\ess f = n$ and $\gap f = p$.
\item \label{itemthmsum2}
There exist functions $g, h \colon A^n \to B$ such that $f = h + g$, $h|_{\Aneq} \equiv 0$, $h \not\equiv 0$, and $\ess g = n - p$.
\end{enumerate}
The decomposition $f = h + g$ given above is unique.
\end{theorem}

We will prove Theorem~\ref{thm:sum} using the following lemma.

\begin{lemma}
\label{lemma:sum}
Assume that $(B; +)$ is a group with neutral element $0$. Let $f \colon A^n \to B$, $n \geq 3$, and $1 \leq p \leq n$. Then the following two conditions are equivalent:
\begin{enumerate}[\indent\rm (a)]
\item \label{itemlemmasum1}
$\ess f = n$ and $\qa f = n - p$.
\item \label{itemlemmasum2}
There exist functions $g, h \colon A^n \to B$ such that $f = h + g$, $h|_{\Aneq} \equiv 0$, $h \not\equiv 0$, and $\ess g = n - p$.
\end{enumerate}
The decomposition $f = h + g$ given above is unique.
\end{lemma}

\begin{proof}
\eqref{itemlemmasum1}$\implies$\eqref{itemlemmasum2}.
Assume that $\ess f = n$ and $\qa f = n - p$. By the definition of quasi-arity, there exists an essentially $(n - p)$-ary support $g \colon A^n \to B$ of $f$. Setting $h := f - g$, we have $f = h + g$. Since $g|_{\Aneq} = f|_{\Aneq}$ by the definition of support, we have that $h|_{\Aneq} \equiv 0$. Furthermore, $h$ is not identically $0$, because otherwise we would have that $f = g$, which constitutes a contradiction to $\ess g = n - p < n = \ess f$.

\eqref{itemlemmasum2}$\implies$\eqref{itemlemmasum1}.
Assume \eqref{itemlemmasum2}.
By Lemma~\ref{lem:qaess}, $\qa f = \ess f|_{\Aneq} = \ess g|_{\Aneq}$, and by Lemma~\ref{lem:diagwitness}, $\ess g|_{\Aneq} = \ess g = n - p$; hence $\qa f = n - p$. Suppose for contradiction that $\ess f < n$, then $\ess f = \qa f = n - p$ by Lemma~\ref{lem:m-ary}. Both $f$ and $g$ are essentially $(\qa f)$-ary supports of $f$; therefore it follows from Proposition~\ref{prop:uniqueSupport} that $f = g$. Thus $h \equiv 0$, which yields a contradiction. 

For the uniqueness of the decomposition $f = h + g$, the function $g$ in the decomposition $f = h + g$ is clearly an essentially $(\qa f)$-ary support of $f$. By the assumption that $\qa f < \ess f$, Proposition~\ref{prop:uniqueSupport} implies that $g$ is uniquely determined, and therefore so is $h$.
\end{proof}

\begin{proof}[Proof of Theorem~\ref{thm:sum}]
Observe that condition \eqref{itemthmsum2} is the same as condition \eqref{itemlemmasum2} of Lemma~\ref{lemma:sum}. The latter is equivalent to \eqref{itemlemmasum1} by Lemma~\ref{lemma:sum}, and \eqref{itemlemmasum1} is equivalent to \eqref{itemthmsum1} by Theorem~\ref{thm:gap}~\eqref{thmitem1}. The uniqueness of the decomposition $f = h + g$ follows from Lemma~\ref{lemma:sum}.
\end{proof}


\section{Functions with arity gap $2$}
\label{sec:gap2}

We prove an analogue of Theorem~\ref{thm:sum} for the case $\gap f = 2$. If $\qa f = n - 2$, then Lemma~\ref{lemma:sum} can be applied, so we only consider the case when $f|_{\Aneq}$ is determined by $\oddsupp$ (see Theorem~\ref{thm:gap}~\eqref{thmitem2}). In this case we cannot expect $f$ to have a support of arity $n-2$, but we may look for a support which is a sum of $(n-2)$-ary functions. We will prove that such a support exists if $B$ is a \emph{Boolean group,} i.e., it is an Abelian group such that $x + x = 0$ holds identically. (However, this is not true for arbitrary groups; this will be discussed in a forthcoming paper~\cite{CLW}.)

First we need to introduce a notation. Let $\varphi \colon A^{n-2} \to B$ be a function that is determined by $\oddsupp$, i.e.,
$\varphi = \varphi^{\ast} \circ \oddsupp$, for some function $\varphi^{\ast} \colon \mathcal{P}(A) \to B$. Let $\widetilde{\varphi}$ be the $n$-ary function defined by
\begin{equation}
\widetilde{\varphi}(x_1, \dots, x_n) = \sum_{\substack{k < n \\ 2 \mid n-k}} \,\, \sum_{1 \leq i_1 < \dots <i_k \leq n} \varphi^{\ast} (\oddsupp (x_{i_1}, \dots, x_{i_k})).
\label{fitilde def}
\end{equation}
Observe that each summand is a variable identification minor of $\varphi$, namely
\[
\varphi^{\ast}(\oddsupp(x_{i_1}, \dots, x_{i_k})) = \varphi (x_{i_1}, \dots, x_{i_k}, y, \dots, y),
\]
where the number of $y$'s is $n - 2 - k$, which is an even number; therefore $y$ is indeed an inessential variable of the function on the right-hand side; moreover, the order of the variables is irrelevant. The function $\widetilde{\varphi}$ is obviously totally symmetric, and according to the following lemma, $\widetilde{\varphi}|_{\Aneq}$ is determined by $\oddsupp$; hence $\widetilde{\varphi}$ is determined by $\oddsupp$ as well according to Remark~\ref{rem:oddsupp2}.

\begin{lemma}
\label{lemma fitilde from fistar}
Assume that $(B; +)$ is a Boolean group with neutral element $0$. Let $\varphi \colon A^{n-2} \to B$ be a function determined by $\oddsupp$. Then for all $(x_1, \dots, x_n) \in \Aneq$ we have
\[
\widetilde{\varphi}(x_1, \dots, x_n) = \varphi^{\ast}(\oddsupp(x_1, \dots, x_n)).
\]
\end{lemma}

\begin{proof}
We have to show that $\widetilde{\varphi}(x_1, \dots, x_n) + \varphi^{\ast}(\oddsupp(x_1, \dots, x_n)) = 0$ holds identically on $\Aneq$. This function differs from the right-hand side of \eqref{fitilde def} only by a summand corresponding to $k = n$:
\begin{multline*}
\widetilde{\varphi}(x_1, \dots, x_n) + \varphi^{\ast}(\oddsupp(x_1, \dots, x_n)) \\
= \sum_{\substack{k \leq n \\ 2 \mid n - k}} \,\, \sum_{1 \leq i_1 < \cdots < i_k \leq n}
\varphi^{\ast}(\oddsupp(x_{i_1}, \dots, x_{i_k})).
\end{multline*}

Let us fix a set $\{a_1, \dots, a_r\} \subseteq A$ and $(x_1, \dots, x_n) \in \Aneq$. We count how many summands there are in the above sum with $\oddsupp(x_{i_1}, \dots, x_{i_k}) = \{a_1, \dots, a_r\}$. If this set occurs at all, then $a_1, \dots, a_r$ can be found among the components of $(x_1, \dots, x_n)$. Let us denote the rest of the elements appearing in $(x_1, \dots, x_n)$ by $a_{r+1}, \dots, a_t$, and for $j = 1, \dots, t$ let $s_j$ stand for the number of occurrences of $a_j$ in $(x_1, \dots, x_n)$. Thus $\{x_1, \dots, x_n\} = \{a_1, \dots, a_t\}$ and $s_1 + \dots + s_t = n$; moreover, $t < n$, because $(x_1, \dots, x_n) \in \Aneq$. If we want to choose $i_1, \dots, i_k$ such that $\oddsupp(x_{i_1}, \dots, x_{i_k}) = \{a_1, \dots, a_r\}$, then we have to choose an odd number of the $s_j$ places occupied by $a_j$ in $(x_1, \dots, x_n)$ for $j = 1, \dots, r$, and an even number of the $s_j$ places occupied by $a_j$ for $j = r + 1, \dots, t$. A set of $s_j$ elements has $2^{s_j - 1}$ subsets with odd cardinality, and likewise $2^{s_j - 1}$ subsets with even cardinality, so the number of possibilities is $2^{s_j - 1}$ in both cases. Thus there are altogether $2^{s_1 - 1} \cdot \ldots \cdot 2^{s_t - 1} = 2^{n-t}$ summands with the same $\oddsupp(x_{i_1}, \dots, x_{i_k})$. This number is even since $t < n$; therefore the terms will cancel each other. This holds for any set $\{a_1, \dots, a_r\}$ and any $(x_1, \dots, x_n) \in \Aneq$; hence $\widetilde{\varphi}(x_1, \dots, x_n) + \varphi^{\ast}(\oddsupp(x_1, \dots, x_n))$ is identically zero on $\Aneq$.
\end{proof}

\begin{theorem}
\label{thm oddsupp existence}
Assume that $(B; +)$ is a Boolean group with neutral element $0$. Let $f \colon A^{n} \to B$ be a function such that $f|_{\Aneq}$ is determined by $\oddsupp$. Then $f$ has a support that is a sum of functions of arity at most $n - 2$.
\end{theorem}

\begin{proof}
Since $f|_{\Aneq}$ is determined by $\oddsupp$, there is a function $f^{\ast} \colon \mathcal{P}(A) \to B$ such that $f|_{\Aneq} = f^{\ast} \circ \oddsupp$. The function $\varphi(x_1, \dots, x_{n-2}) := f(x_1, \dots, x_{n-2}, y, y)$ is determined by $\oddsupp$, and we can suppose that the corresponding function $\varphi^{\ast}$ coincides with $f^{\ast}$, since
\[
\varphi(x_1, \dots, x_{n-2}) = f(x_1, \dots, x_{n-2}, y, y) = f^{\ast}(\oddsupp(x_1, \dots, x_{n-2}))
\]
for all $(x_1, \dots, x_{n-2}) \in A^{n-2}$. Applying Lemma~\ref{lemma fitilde from fistar} we get the following equality for every $(x_1, \dots, x_n) \in \Aneq$:
\begin{align*}
\widetilde{\varphi}(x_1, \dots, x_n) &  = \varphi^{\ast}(\oddsupp(x_1, \dots, x_n)) \\
& = f^{\ast}(\oddsupp(x_1, \dots, x_n)) = f(x_1, \dots, x_n).
\end{align*}
This shows that $\widetilde{\varphi}$ is a support of $f$, and from \eqref{fitilde def} it is clear that $\widetilde{\varphi}$ is a sum of at most $(n-2)$-ary functions.
\end{proof}

\begin{remark}
\label{rem n>|A|}
Let us note that if $A$ is finite and $n > \card{A}$, then $A^n = \Aneq$; hence the only support of $f$ is $f$ itself. In this case the above theorem implies that $f$ itself can be expressed as a sum of functions of arity at most $n-2$.
\end{remark}

Next we prove a uniqueness companion to the above theorem. Here we do not need the assumption that $B$ is a Boolean group: if there exists a support that is a sum of at most $(n-2)$-ary functions, then it is unique for any Abelian group $B$. Note that this does not exclude the possibility that this unique support can be written in more than one way as a sum of at most $(n-2)$-ary functions. Observe also that the following theorem generalizes Proposition~\ref{prop:uniqueSupport} in the case $m = n - 2$.

\begin{theorem}
\label{oddsupp uniqueness}
Assume that $(B;+)$ is an Abelian group with neutral element $0$. Then a function $f \colon A^n \to B$ can have at most one support that is a sum of functions of arity at most $n-2$.
\end{theorem}

\begin{proof}
Suppose that $g_1$ and $g_2$ are supports of $f$ and both of them can be expressed as sums of at most $(n-2)$-ary functions. Then $g = g_1  - g_2$ is also a sum of at most $(n-2)$-ary functions, and $g|_{\Aneq}$ is constant zero. Let us choose the smallest $k$ such that $g$ can be written as a sum of functions of arity at most $k$. If $k = 0$, then $g$ is constant; hence $g = 0$ and then we can conclude that $g_1 = g_2$. To complete the proof, we just have to show that the assumption $1 \leq k \leq n-2$ leads to a contradiction.

In the expression of $g$ as a sum of at most $k$-ary functions we can combine functions depending on the same set of variables to a single function, and by introducing dummy variables we can make all of the summands $n$-ary functions. Then $g$ takes the following form:
\begin{equation}
g(x_1, \dots, x_n) = \sum_I g_I(x_1, \dots, x_n),
\label{g from  g_I}
\end{equation}
where $I$ ranges over the $k$-element subsets of $\{1, 2, \dots, n\}$, and $g_I \colon A^n \to B$ is a function which only depends on some of the variables $x_i$ ($i \in I$). Let us choose a constant $c \in A$ and substitute this into the last $n - k$ variables. Since $n - k \geq 2$, the resulting vector will lie in $\Aneq$; hence the value of $g$ will be zero:
\[
0 = g(x_1, \dots, x_k, c, \dots, c) = \sum_I g_I(x_1, \dots, x_k, c, \dots, c).
\]

Let $J = \{1, \dots, k\}$, and let us express $g_J$ from the above equation:
\[
g_J(x_1, \dots, x_n) = g_J(x_1, \dots, x_k, c, \dots, c) = - \sum_{I \neq J} g_I(x_1, \dots, x_k, c, \dots, c).
\]
For each $k$-element subset $I$ of $\{1, 2, \dots, n\}$, the function $g_I(x_1, \dots, x_k, c, \dots, c)$ depends only on the variables $x_i$ ($i \in I \cap J$); thus its essential arity is at most $k-1$ whenever $I$ is different from $J$. This means that the above expression for $g_J$ can be regarded as a sum of at most $(k-1)$-ary functions (after getting rid of the dummy variables). We can get a similar expression for $g_J$ for any $k$-element subset $J$ of $\{1, 2, \dots, n\}$, and substituting these into \eqref{g from g_I} we see that $g$ is a sum of at most $(k-1)$-ary functions. This contradicts the minimality of $k$, which shows that $k \geq 1$ is indeed impossible.
\end{proof}

Combining the above results with Theorem~\ref{thm:gap} and Lemma~\ref{lemma:sum} we get the characterization of functions $f \colon A^n \to B$ with $\gap f = 2$ for the case when $B$ is a Boolean group.

\begin{theorem}
\label{thm oddsupp}
Assume that $(B;+)$ is a Boolean group with neutral element $0$. Let $f \colon A^n \to B$ be a function of arity at least $4$. Then the following two conditions are equivalent:

\begin{enumerate}[\indent\rm (1)]
\item \label{thm oddsupp 1}
$\ess f = n$ and $\gap f = 2$.

\item \label{thm oddsupp 2}
There exist functions $g, h \colon A^n \to B$ such that $f = h + g$, $h|_{\Aneq} \equiv 0$, and either

\begin{enumerate}[\rm (a)]
\item \label{thm oddsupp 2a}
$\ess g = n - 2$ and $h \not\equiv 0$, or

\item \label{thm oddsupp 2b}
$g = \widetilde{\varphi}$ for some nonconstant $(n-2)$-ary function $\varphi$ that is determined by $\oddsupp$.
\end{enumerate}
\end{enumerate}
The decomposition $f = h + g$ given above is unique.
\end{theorem}

\begin{proof}
The uniqueness follows immediately from Theorem~\ref{oddsupp uniqueness}, so we just need to show that \eqref{thm oddsupp 1} and \eqref{thm oddsupp 2} are equivalent.

\eqref{thm oddsupp 1}$\implies$\eqref{thm oddsupp 2}.
By Theorem~\ref{thm:gap}~\eqref{thmitem2} we have two cases: either $\qa f = n - 2$, or $\qa f = n$ and $f|_{\Aneq}$ is determined by $\oddsupp$. In the first case Lemma~\ref{lemma:sum} shows that \eqref{thm oddsupp 2a} holds. In the second case we apply Theorem~\ref{thm oddsupp existence} to find an $(n-2)$-ary function $\varphi$ such that $g = \widetilde{\varphi}$ is a support of $f$, and we let $h = f + g$. If $\varphi$ is constant, then so is $\widetilde{\varphi}$, and then $f|_{\Aneq}$ is constant as well, contradicting that $\qa f = n$.

\eqref{thm oddsupp 2}$\implies$\eqref{thm oddsupp 1}.
The case \eqref{thm oddsupp 2a} is settled by Lemma~\ref{lemma:sum} and Theorem~\ref{thm:gap}~\eqref{thmitem2}, so let us assume that \eqref{thm oddsupp 2b} holds. It is clear that $f|_{\Aneq}$ is determined by $\oddsupp$, so according to Theorem~\ref{thm:gap} it suffices to show that $\qa f = \ess f = n$. The function $f|_{\Aneq} = \widetilde{\varphi}|_{\Aneq}$ is totally symmetric, hence it either depends on all of its variables, or on one of them, i.e., either $\qa f = n$ or $\qa f = 0$. In the first case we are done, since $\ess f$ cannot be less than $\qa f$. In the second case Lemma~\ref{lemma fitilde from fistar} implies that $\varphi^{\ast}$ takes on the same value for every subset of $A$ of size $n - 2, n - 4, \dots$. Since only these values of $\varphi^{\star}$ are relevant for determining $\varphi = \varphi^{\ast} \circ \oddsupp$, we can conclude that $\varphi$ is constant, contrary to our assumption.
\end{proof}


\section{The number of finite functions with a given arity gap}
\label{sec:enumerate}

The classification of functions according to their arity gap (Theorem~\ref{thm:gap}) and the unique decompositions of functions provided by Theorem~\ref{thm:sum} can be applied to count, for finite sets $A$ and $B$, and for each $n$ and $p$ the number of functions $f \colon A^n \to B$ with $\gap f = p$.

For positive integers $m$, $i$, we will denote by $(m)_i$ the \emph{falling factorial}
\[
(m)_i := m (m - 1) \cdots (m - (i - 1)).
\]
Note that if $i > m$, then $(m)_i = 0$, because one of the factors in the above expression is $0$.

Let $A$ and $B$ be finite sets with $\card{A} = k$, $\card{B} = \ell$. Let us denote by $G^{k \ell}_{n p}$ the number of functions $f \colon A^n \to B$ with $\ess f = n$ and $\gap f = p$, and let us denote by $Q^{k \ell}_{n m}$ the number of functions $f \colon A^n \to B$ with $\ess f = n$ and $\qa f = m$.

It is well-known (see Wernick~\cite{Wernick}) that the number of functions $g \colon A^n \to B$ that depend on exactly $r$ variables ($0 \leq r \leq n$) is
\[
U^{k \ell}_{nr} := \binom{n}{r} \sum_{i = 0}^r (-1)^i \binom{r}{i} \ell^{k^{r-i}}.
\]
The number of functions $h \colon A^n \to B$ such that $h|_{\Aneq} \equiv 0$, $h \not\equiv 0$ is
\[
V^{k \ell}_n := \ell^{(k)_n} - 1.
\]

\begin{lemma}
\label{lemma:Qklnm}
For $k \geq 2$, $\ell \geq 2$, $n \geq 3$,
\begin{equation}
Q^{k \ell}_{n m} =
\begin{cases}
U^{k \ell}_{n m} V^{k \ell}_n, & \text{if $m < n$,} \\
U^{k \ell}_{nn} \ell^{(k)_n} - V^{k \ell}_n \ell^{k^n}, & \text{if $m = n$.}
\end{cases}
\label{eq:Qklnn-p}
\end{equation}
\end{lemma}

\begin{proof}
By Lemma~\ref{lemma:sum}, for $3 \leq n \leq k$ and $m < n$,
\[
Q^{k \ell}_{nm} = U^{k \ell}_{nm} V^{k \ell}_n.
\]
If $n > k$, then $V^{k \ell}_n = 0$ and hence the right-hand side of the above equation is $0$ as well. Indeed, $Q^{k \ell}_{nm} = 0$ in this case, because for $f \colon A^n \to B$, $\qa f = \ess f$ whenever $n > k$.

Consider then the case that $m = n$. By the above formula, we have
\begin{equation}
Q^{k \ell}_{nn}
= U^{k \ell}_{nn} - \sum_{i = 0}^{n-1} Q^{k \ell}_{ni}
= U^{k \ell}_{nn} - \sum_{i = 0}^{n-1} U^{k \ell}_{ni} V^{k \ell}_n
= U^{k \ell}_{nn} - V^{k \ell}_n \sum_{i = 0}^{n-1} U^{k \ell}_{ni}.
\label{eq:Qklnn}
\end{equation}
The sum $\sum_{i = 0}^{n-1} U^{k \ell}_{ni}$ counts the number of functions $f \colon A^n \to B$ with $\ess f < n$; hence
\[
\sum_{i = 0}^{n-1} U^{k \ell}_{ni}
= \ell^{k^n} - U^{k \ell}_{nn}.
\]
Substituting this back to \eqref{eq:Qklnn}, we have
\begin{multline*}
Q^{k \ell}_{nn}
= U^{k \ell}_{nn} - V^{k \ell}_n (\ell^{k^n} - U^{k \ell}_{nn})
= U^{k \ell}_{nn}(1 + V^{k \ell}_n) - V^{k \ell}_n \ell^{k^n}
= U^{k \ell}_{nn} \ell^{(k)_n} - V^{k \ell}_n \ell^{k^n}.
\qedhere
\end{multline*}
\end{proof}

Let us denote by $O^{k \ell}_n$ the number of functions $f \colon A^n \to B$ such that $\ess f = n$, $\qa f = n$ and $f|_{\Aneq}$ is determined by $\oddsupp$.

\begin{lemma}
For $k \geq 2$, $\ell \geq 2$, $n \geq 2$,
\label{lem:Olmn}
\[
O^{k \ell}_n =
\begin{cases}
\ell^{2^{k - 1}} - \ell, & \text{if $n > k$,} \\
\ell^{(k)_n} (\ell^{S^k_n} - \ell), & \text{if $n \leq k$,}
\end{cases}
\]
where
\begin{equation}
S^k_n =
\begin{cases}
\sum_{i = 0}^{\frac{n}{2} - 1} \binom{k}{2i}, & \text{if $n$ is even,} \\
\sum_{i = 0}^{\frac{n - 1}{2} - 1} \binom{k}{2i + 1}, & \text{if $n$ is odd.}
\end{cases}
\label{eq:Skn}
\end{equation}
\end{lemma}

\begin{proof}
Let $f \colon A^n \to B$ be a map such that $f|_{\Aneq}$ is determined by $\oddsupp$. It is clear that then $f|_{\Aneq}$ is totally symmetric; hence, either all variables are essential in $f|_{\Aneq}$ or none of them is. In the former case, $\qa f = n$, and in the latter case $\qa f = 0$ (i.e., $f|_{\Aneq}$ is constant). Therefore $O^{k \ell}_n$ equals the number of nonconstant maps $\range \oddsupp|_{\Aneq} \to B$ multiplied by the number of maps $A^n \setminus \Aneq \to B$. By Remark~\ref{rem:oddsupp},
\[
\range \oddsupp|_{\Aneq} = \{S \subseteq A : \card{S} \equiv n \!\!\!\!\pmod{2},\, \card{S} \leq n - 2\}.
\]

Consider first the case that $n > k$. Then $\Aneq = A^n$ and there is only one map $A^n \setminus \Aneq \to B$, namely the empty map. In this case, $\range \oddsupp|_{\Aneq}$ equals the set of odd subsets of $A$ or the set of even subsets of $A$, depending on the parity of $n$. It is well-known that the number of odd subsets of $A$ equals the number of even subsets of $A$, and this number is $2^{k - 1}$. Thus $O^{k \ell}_n$ equals the number of nonconstant functions from the set of even (or odd) subsets of $A$ to $B$, which is $\ell^{2^{k - 1}} - \ell$. Note that this number does not depend on $n$.

Consider then the case that $n \leq k$. If $n = 2q$, then
\[
\card{\range \oddsupp|_{\Aneq}} = \sum_{i = 0}^{q-1} \binom{k}{2i}.
\]
If $n = 2q + 1$, then
\[
\card{\range \oddsupp|_{\Aneq}} = \sum_{i = 0}^{q-1} \binom{k}{2i + 1}.
\]
The number of maps $A^n \setminus \Aneq \to B$ is $\ell^{(k)_n}$. Thus,
\[
O^{k \ell}_n = \ell^{(k)_n} (\ell^{S^k_n} - \ell),
\]
where $S^k_n$ is as given in equation~\eqref{eq:Skn}.
\end{proof}

\begin{theorem} Let $k \geq 2$, $\ell \geq 2$, $n \geq 2$.
\begin{enumerate}[\rm (i)]
\item\label{count1} If $n > k$ and $3 \leq p \leq n$, then $G^{k \ell}_{np} = 0$.

\item\label{count2} If $n > k$ and $n \geq 4$, then
\[
G^{k \ell}_{n2} = O^{k \ell}_n = \ell^{2^{k - 1}} - \ell,
\qquad
G^{k \ell}_{n1} = U^{k \ell}_{nn} - G^{k \ell}_{n2}.
\]

\item\label{count3} If $3 \leq n \leq k$ and $3 \leq p \leq n$, then $G^{k \ell}_{np} = U^{k \ell}_{n (n-p)} V^{k \ell}_n$.

\item\label{count4} If $4 \leq n \leq k$, then 
\[
G^{k \ell}_{n2}
= U^{k \ell}_{n(n-2)} V^{k \ell}_n + O^{k \ell}_n,
\qquad
G^{k \ell}_{n1}
= U^{k \ell}_{n(n-1)} V^{k \ell}_n + U^{k \ell}_{nn} \ell^{(k)_n} - V^{k \ell}_n \ell^{k^n} - O^{k \ell}_n
\]

\item\label{count5} $G^{k \ell}_{32} = (8 \ell^{(k)_3} - 3) (\ell^k - \ell)$, $G^{k \ell}_{31} = U^{k \ell}_{33} - G^{k \ell}_{33} - G^{k \ell}_{32}$.

\item\label{count6} $G^{k \ell}_{22} = \ell^{(k)_2 + 1} - \ell$, $G^{k \ell}_{21} = U^{k \ell}_{22} - G^{k \ell}_{22}$.
\end{enumerate}
\end{theorem}

\begin{proof}
\eqref{count1} Follows from Theorem~\ref{Willard1.2}.

\eqref{count2} If $f \colon A^n \to B$ depends on all of its variables and $n > k$, then by Remark~\ref{rem:an} $\qa f = \ess f = n$. Thus $\gap f = 2$ if and only if $f|_{\Aneq} = f$ is determined by $\oddsupp$. Thus, $G^{k \ell}_{n2} = O^{k \ell}_n = \ell^{2^{k - 1}} - \ell$ by Lemma~\ref{lem:Olmn}. The equality for $G^{k \ell}_{n1}$ follows immediately from~\eqref{count1} and the equality for $G^{k \ell}_{n2}$.

\eqref{count3} By Theorem~\ref{thm:gap}~\eqref{thmitem1}, for $3 \leq n \leq k$ and $3 \leq p \leq n$, we have $G^{k \ell}_{n p} = Q^{k \ell}_{n (n-p)}$, and $Q^{k \ell}_{n (n-p)} = U^{k \ell}_{n (n-p)} V^{k \ell}_n$ by Lemma~\ref{lemma:Qklnm}.

\eqref{count4} By Theorem~\ref{thm:gap}, and Lemma~\ref{lemma:Qklnm}, for $n \geq 4$, we have
\[
G^{k \ell}_{n2}
= Q^{k \ell}_{n(n-2)} + O^{k \ell}_n
= U^{k \ell}_{n(n-2)} V^{k \ell}_n + O^{k \ell}_n
\]
and
\[
G^{k \ell}_{n1}
= Q^{k \ell}_{n(n-1)} + Q^{k \ell}_{nn} - O^{k \ell}_n
= U^{k \ell}_{n(n-1)} V^{k \ell}_n + U^{k \ell}_{nn} \ell^{(k)_n} - V^{k \ell}_n \ell^{k^n} - O^{k \ell}_n.
\]

\eqref{count5}
We apply Theorem~\ref{thm:gap}~\eqref{thmitem3} in order to determine $G^{k \ell}_{32}$. It is easy to verify that given nonconstant functions $h, h' \colon A \to B$, elements $i_1, i_2, i_3, i'_1, i'_2, i'_3 \in \{0, 1\}$ and functions $f, f' \colon A^3 \to B$ such that
\begin{align*}
f(x_1, x_0, x_0) &= h(x_{i_1}), &
f(x_0, x_1, x_0) &= h(x_{i_2}), &
f(x_0, x_0, x_1) &= h(x_{i_3}) \\
f'(x_1, x_0, x_0) &= h'(x_{i'_1}), &
f'(x_0, x_1, x_0) &= h'(x_{i'_2}), &
f'(x_0, x_0, x_1) &= h'(x_{i'_3}),
\end{align*}
it holds that $f|_{\Aneq[3]} = f'|_{\Aneq[3]}$ if and only if $h = h'$, $i_1 = i'_1$, $i_2 = i'_2$, $i_3 = i'_3$.

There are $2^3 = 8$ choices for $(i_1, i_2, i_3)$, there are $\ell^k - \ell$ nonconstant maps $h \colon A \to B$, and there are $\ell^{(k)_3}$ ways to choose values for a function on $A^3 \setminus \Aneq[3]$. Thus the number of functions of the form given in Theorem~\ref{thm:gap}~\eqref{thmitem3} is
\[
8 (\ell^k - \ell) \ell^{(k)_3}.
\]
However, some of the functions corresponding to Theorem~\ref{thm:gap}~\eqref{thmitem3} are not essentially ternary, and we have to subtract the number of these functions from the above number. We claim that $f \colon A^3 \to B$ satisfies the condition of Theorem~\ref{thm:gap}~\eqref{thmitem3} and $\ess f < 3$ if and only if $\ess f = 1$. For, every essentially unary function $f \colon A^3 \to B$ satisfies the condition of Theorem~\ref{thm:gap}~\eqref{thmitem3} with $(i_1, i_2, i_3) \in \{(1,0,0), (0,1,0), (0,0,1)\}$ and $h(x) = f(x,x,x)$. Conversely, suppose that $f$ satisfies the condition of Theorem~\ref{thm:gap}~\eqref{thmitem3} and $\ess f < 3$, say, the last variable of $f$ is inessential. Then we have
\[
f(x_0, x_1, x_2) = f(x_0, x_1, x_0) = h(x_{i_2}),
\]
i.e., $f$ is equivalent to the nonconstant unary function $h$.

The number of essentially unary ternary functions is $3 (\ell^k - \ell)$; hence
\[
G^{k \ell}_{32}
= 8 (\ell^k - \ell) \ell^{(k)_3} - 3(\ell^k - \ell)
= (8 \ell^{(k)_3} - 3) (\ell^k - \ell).
\]

It is clear that
\[
G^{k \ell}_{31} = U^{k \ell}_{33} - G^{k \ell}_{33} - G^{k \ell}_{32}.
\]

\eqref{count6} For $f \colon A^2 \to B$, $\gap f = 2$ if and only if $f|_{\Aneq[2]}$ is constant (but $f$ itself is not constant). Thus $G^{k \ell}_{22} = \ell^{(k)_2 + 1} - \ell$. It is clear that $G^{k \ell}_{21} = U^{k \ell}_{22} - G^{k \ell}_{22}$.
\end{proof}

We have evaluated $G^{k \ell}_{np}$ for some values of $k$, $\ell$, $n$, $p$ in Table~\ref{table:Gklnp}.

\begin{table}
\begin{tabular}{|c|c|c|r|r|r|r|r|r|}
\hline
$k$ & $\ell$ & $n$ & \multicolumn{1}{c|}{$U^{k \ell}_{nn}$} & \multicolumn{1}{c|}{$G^{k \ell}_{n1}$} & \multicolumn{1}{c|}{$G^{k \ell}_{n2}$} & \multicolumn{1}{c|}{$G^{k \ell}_{n3}$} & \multicolumn{1}{c|}{$G^{k \ell}_{n4}$} & \multicolumn{1}{c|}{$G^{k \ell}_{n5}$} \\
\hline
$2$ & $2$ & $2$ &         $10$ &          $4$ &  $6$ & --- & --- & --- \\
    &     & $3$ &        $218$ &        $208$ & $10$ & $0$ & --- & --- \\
    &     & $4$ &      $64594$ &      $64592$ &  $2$ & $0$ & $0$ & --- \\
    &     & $5$ & $4294642034$ & $4294642032$ &  $2$ & $0$ & $0$ & $0$ \\
\hline
$3$ & $3$ & $2$ &              $19632$ &              $17448$ &   $2184$ &    --- & --- & --- \\
    &     & $3$ &      $7625597426016$ &      $7625597283936$ & $139896$ & $2184$ & --- & --- \\
    &     & $4$ &  $4.4 \cdot 10^{38}$ &  $4.4 \cdot 10^{38}$ &     $78$ &    $0$ & $0$ & --- \\
    &     & $5$ & $8.7 \cdot 10^{115}$ & $8.7 \cdot 10^{115}$ &     $78$ &    $0$ & $0$ & $0$ \\
\hline
$4$ & $4$ & $2$ &         $4294966788$ &         $4227857928$ &          $67108860$ &                 --- & --- & --- \\
    &     & $3$ &  $3.4 \cdot 10^{38}$ &  $3.4 \cdot 10^{38}$ & $5.7 \cdot 10^{17}$ & $1.1 \cdot 10^{15}$ & --- & --- \\
    &     & $4$ & $1.3 \cdot 10^{154}$ & $1.3 \cdot 10^{154}$ & $7.3 \cdot 10^{24}$ & $2.8 \cdot 10^{17}$ & $1.1 \cdot 10^{15}$ & --- \\
    &     & $5$ & $3.2 \cdot 10^{616}$ & $3.2 \cdot 10^{616}$ &             $65532$ &                 $0$ & $0$ & $0$ \\
\hline    
\end{tabular}
\bigskip
\caption{$G^{k \ell}_{np}$ for small values of $k$, $\ell$, $n$, $p$.}
\label{table:Gklnp}
\end{table}


\section*{Acknowledgements}

An idea of a decomposition scheme of functions based on the arity gap and the problem of counting the number of operations $f \colon A^n \to A$ with $\ess f = n$ and $\gap = p$, for every $n$ and $p$, were presented by Shtrakov and Koppitz~\cite{SK}. However, the results presented in~\cite{SK} seem to be inconclusive or incorrect.


\end{document}